\UseRawInputEncoding
\documentclass[reqno]{amsart}
\usepackage{amssymb}
\usepackage{xcolor}

\usepackage[colorlinks=true]{hyperref}
\hypersetup{urlcolor=blue, citecolor=red}

\newtheorem{Lemma}{Lemma}[section]
\newtheorem{Theorem}[Lemma]{Theorem}
\newtheorem{Proposition}[Lemma]{Proposition}
\newtheorem{Corollary}[Lemma]{Corollary}
\newtheorem{Remark}[Lemma]{Remark}
\newtheorem{Definition}[Lemma]{Definition}

\def\bfG{\mathbf G}

\def\R{\mathbb R}
\def\C{\mathbb C}
\def\Z{\mathbb Z}

\newcommand{\p}{\partial}
\newcommand{\<}{\langle}
\renewcommand{\>}{\rangle}
\renewcommand{\div}{\operatorname{div}}
\newcommand{\Vol}{\operatorname{Vol}}
\newcommand{\Id}{\operatorname{Id}}
\newcommand{\Hess}{\operatorname{Hess}}

\begin{document}
\title[The attenuated thermostat ray transforms]{The attenuated ray transforms on Gaussian thermostats with negative curvature}

\author[Yernat M. Assylbekov]{Yernat M. Assylbekov}
\address{Department of Computational and Applied Mathematics, Rice University, Houston, TX 77005, USA}
\email{yernat.assylbekov@gmail.com}
\author[Franklin T. Rea]{Franklin T. Rea}
\address{Department of Mathematics, University of North Carolina, Chapel Hill, NC 27514, USA}
\email{franklin.t.rea@gmail.com}

\begin{abstract}
In the present paper we consider a Gaussian thermostat on a compact Riemannian surface with negative thermostat curvature. In the case of surfaces with boundary, we show that the thermostat ray transform with attenuation given by a general connection and Higgs field is injective, modulo the natural obstruction, for tensors. We also prove that the connection and Higgs field can be determined, up to a gauge transformation, from the knowledge of the parallel transport between boundary points along all possible thermostat geodesics. In the case of closed surfaces, we obtain similar results for the ray transform with some additional conditions on the connection. Under the same condition, we study connections and Higgs fields whose parallel transport along periodic thermostat geodesics coincides with the ones for the flat connection.
\end{abstract}

\maketitle

%--------------------------------------------------------------------------------------------------------------------------------------------------------------------

\section{Introduction and statement of results}\label{sec:: intro}

\subsection{Gaussian Thermostats}
Let $(M,g)$ be a compact oriented Riemannian manifold and let $E$ be a smooth vector field on $M$ which we refer to as the \emph{external field}. We say that a parameterized curve $\gamma(t)$ on $M$ is a \emph{thermostat geodesic} if it satisfies the equation
\begin{equation}\label{eqn::Gaussian thermostat}
D_t\dot \gamma=E(\gamma)-\frac{\<E(\gamma),\dot\gamma\>}{|\dot\gamma|^2}\dot\gamma,
\end{equation}
where $D_t$ is the covariant derivative along $\gamma$. This differential equation defines a flow $\phi_t=(\gamma(t),\dot{\gamma}(t))$ on the unit sphere bundle $SM$ which is called a \emph{Gaussian thermostat}. We will also refer to $\phi_t$ as the \emph{thermostat flow}. The flow $\phi_t$ reduces to the geodesic flow when $E=0$. As in the case of geodesic flows, thermostat flows are reversible in the sense that the flip $(x, v)\mapsto(x,-v)$ conjugates $\phi_t$ with $\phi_{-t}$. We denote the Gaussian thermostat by $(M,g,E)$ and the generating vector field of $\phi_t$ by $\bfG_E$.

Gaussian thermostats were first introduced in \cite{hoover1986molecular} by Hoover who studied a class of dynamical systems in mechanics with thermostatic term satisfying \emph{Gauss' least Constraint Principle} for nonholonomic constraints. It is known that such constraints do not have variational principles in general.  Interesting models in non-equilibrium statistical mechanics are provided by Gaussian thermostats \cite{gallavotti1997srb,gallavotti1999new,ruelle1999smooth}. In this model, temperature fluctuations prevented by maintaining a constant kinetic energy. More precisely, the external field $E$ in \eqref{eqn::Gaussian thermostat} contributes to the kinetic energy. But the second term in \eqref{eqn::Gaussian thermostat} is the ``thermostat", which damps the component of $E$ parallel to the velocity of the trajectory, and, therefore, retains the kinetic energy to be constant. For this reason, Gaussian thermostats are also known as \emph{isokinetic dynamics}.

Gaussian thermostats also arise in Weyl geometry. More precisely, when the geodesics of a Weyl manifold are reparameterized with respect to the arc-length, they turn out to be identical with thermostat geodesics; see \cite{przytycki2008gaussian}. Weyl geometry has been actively studied by many authors in mathematics \cite{calderbank2001einstein,folland1970weyl,higa1993weyl,mettler2018convex} with applications to static Yang-Mills-Higgs theory and twistor theory \cite{hitchin1982complex,hitchin1982monopoles,jones1985minitwistor}. See also \cite{scholz2018unexpected} for the applications of the space-time version of Weyl geometry in the theories of gravity, quantum mechanics, elementary particle physics, and cosmology. 

More recently, some inverse problems for Gaussian thermostats were addressed in \cite{assylbekov2018general,assylbekov2017invariant,dairbekov2007entropy,dairbekov2007entropy2}. Dynamical and geometrical properties were studied in \cite{assylbekov2014hopf,mettler2018convex,mettler2019holomorphic,paternain2006regularity, wojtkowski2000magnetic,wojtkowski2000w}.

In the present paper we study several geometric inverse problems for Gaussian thermostats when $M$ is a surface. In this setting, the equation \eqref{eqn::Gaussian thermostat} can be rewritten as
\begin{equation}\label{def::thermostat 2D and lambda}
D_t\dot\gamma=\lambda(\gamma,\dot\gamma)i\dot\gamma,\qquad \lambda(x,v):=\<E(x),iv\>.
\end{equation}

\begin{Definition}{\rm For a given Gaussian thermostat $(M,g,E)$, the \emph{thermostat curvature} is the quantity $K_E:= K-\div_g E$, where $K$ is the Gaussian curvature of the surface $(M,g)$.
}\end{Definition}

We note that $K_E$ is precisely the sectional curvature of the Weyl connection in two dimensions \cite{wojtkowski2000w}. We will mainly focus on Gaussian thermostats with negative thermostat curvature.% It is easy to see that Gaussian thermostats on compact surfaces with negative thermostat curvature generalize compact Riemannian surfaces without conjugate points; see Appendix~\ref{sec::appendix}.

\subsection{Connections and Higgs fields}
We define a \emph{connection} on the trivial bundle $M\times \C^n$ as a matrix-valued smooth map $A:TM\to \mathfrak{gl}(n,\C)$ which is linear in $v\in T_x M$ for a fixed $x\in M$, and define a \emph{Higgs field} as a matrix-valued smooth map $\Phi:M\to \mathfrak{gl}(n,\C)$. The connection $A$ induces a covariant derivative which acts on sections of $M\times \C^n$ by $d_A:=d+A$. Pairs of connections and Higgs fields $(A,\Phi)$ are very important in the Yang-Mills-Higgs theories; see \cite{dunajski2010solitons,hitchin1999integrable,manton2004topological,mason1996integrability}. It is natural to consider connections and Higgs fields for Gaussian thermostats, since the latter, as mentioned before, is also related to Yang-Mills-Higgs theories via Weyl geometry.

\subsection{Main results on surfaces with boundary}\label{sec::surfaces with boundary}
Let $(M,g)$ be a smooth compact and oriented Riemannian surface with boundary $\p M$ and let $E$ be an external field. We consider the set of inward and outward unit vectors on the boundary of $M$ defined as
$$
\p_\pm SM=\{(x,v)\in SM:x\in\p M,\pm\langle v,\nu(x)\rangle\ge0\},
$$
where $\nu$ is the inward unit normal to $\p M$. The thermostat geodesics entering $M$ can be parametrized by $\p_+SM$. We assume that $(M,g,E)$ is \emph{non-trapping}, i.e. for any $(x,v)\in SM$ the first non-negative time $\tau(x,v)$ when the thermostat geodesic $\gamma_{x,v}$, with $x=\gamma_{x,v}(0)$, $v=\dot\gamma_{x,v}(0)$, exits $M$ is finite.

%\subsubsection*{\bf Inverse problems for connections and Higgs fields}
For the rest of the paper, if $\p M$ is non-empty, we will work with \emph{strictly convex} Gaussian thermostats, i.e. the ones satisfying %the boundary $\p M$ is strictly convex w.r.t. thermostat geodesics.
$$
\Lambda(x,v)>\<E(x),\nu(x)\>_{g(x)}\quad\text{for all}\quad(x,v)\in S(\p M),
$$
where $\Lambda$ is the second fundamental form of $\p M$; see Appendix~\ref{sec::convexity} for explanation and more details.

For a pair of connection and Higgs field $(A,\Phi)$ on the trivial bundle $M\times\C^n$, we have a scattering data defined as follows. Consider the unique solution $U_{A,\Phi}:SM\to GL(n,\C)$ for the transport equation
\begin{equation}\label{eqn::transport equation for U}
(\bfG_E+A+\Phi)U_{A,\Phi}=0\quad\text{in}\quad SM,\qquad U_{A,\Phi}|_{\p_+ SM}=\Id.
\end{equation}
\begin{Definition}{\rm
The \emph{scattering data} of a pair of connection and Higgs field $(A,\Phi)$ is the map $C_{A,\Phi}:\p_- SM\to GL(n,\C)$ defined as
$$
C_{A,\Phi}(x,v):=U_{A,\Phi}(x,v),\quad (x,v)\in \p_- SM.
$$
}\end{Definition}

We are interested in the inverse problem of determining $(A,\Phi)$ from the knowledge of $C_{A,\Phi}$. However, there is a natural gauge invariance for the scattering data $C_{A,\Phi}$. Indeed, suppose that $Q\in C^\infty(M;GL(n,\C))$ such that $Q|_{\p M}=\Id$ and $U_{A,\Phi}$ satisfies \eqref{eqn::transport equation for U}. Then, using the fact $d(Q^{-1})Q=-Q^{-1}dQ$, it is not difficult to see that $Q^{-1}U_{A,\Phi}$ satisfies
$$
\big(\bfG_E+Q^{-1}(d+A)Q+Q^{-1}\Phi Q\big)(Q^{-1}U_{A,\Phi})=0\text{ in }SM,\quad Q^{-1}U_{A,\Phi}|_{\p_+ SM}=\Id.
$$
Thus, we have
$$
C_{Q^{-1}(d+A)Q,Q^{-1}\Phi Q}=(Q^{-1}U_{A,\Phi})|_{\p_- SM}=U_{A,\Phi}|_{\p_- SM}=C_{A,\Phi}.
$$
Therefore, we can only hope to recover $(A,\Phi)$ from $C_{A,\Phi}$ up to such an obstruction. Our first main result says that it is possible when the thermostat curvature $K_E$ is negative.

\begin{Theorem}\label{th::main 2}
Let $(M,g,E)$ be a strictly convex and non-trapping Gaussian thermostat on a compact oriented surface with boundary such that $K_E<0$. Let also $A,B:TM\to \mathfrak{gl}(n,\C)$ be two connections and $\Phi,\Psi:M\to \mathfrak{gl}(n,\C)$ be two Higgs fields. If $C_{A,\Phi}=C_{B,\Psi}$, then $B=Q^{-1}(d+A)Q$ and $\Psi=Q^{-1}\Phi Q$ for some $Q\in C^\infty(M;GL(n,\C))$ such that $Q|_{\p M}=\Id$.
\end{Theorem}

Similar result was earlier established in \cite{paternain2018carleman} for geodesic flows in arbitrary dimensions. For earlier related results in the case of geodesic flows, see \cite{eskin2004nonabelian,finch2001connection,guillarmou2016negconnections,novikov2002determination, paternain2012attenuated,paternain2016geodesic,sharafutdinov2000inverse,zhou2017generic}. In the presence of a magnetic field some positive result is given in \cite{ainsworth2013attenuated}.

The proof is based on the reduction of the problem to an integral geometry problem. This motivates us to study attenuated ray transforms along thermostat geodesics which is the central object of the paper.
%\subsubsection*{\bf The attenuated ray transforms for connections and Higgs fields}
%Our  is the attenuated ray transform along thermostat geodesics which we now define.

Given $f\in C^\infty(SM;\C^n)$, consider the following transport equation for $u:SM\to \C^n$
$$
(\mathbf G_E+A+\Phi)u=-f\quad\text{in}\quad SM,\qquad u\big|_{\p_- SM}=0.
$$
Here $A$ and $\Phi$ act on functions on $SM$ by matrix multiplication. This equation has a unique solution $u^f$, since on any fixed thermostat geodesic the transport equation is a linear system of ordinary differential equations with zero initial condition.
\begin{Definition}{\rm
The \emph{\bfseries attenuated thermostat ray transform} $I_{A,\Phi} f$ of $f\in C^\infty(SM;\C^n)$, with attenuation given by a connection $A:TM\to \mathfrak{gl}(n,\C)$ and a Higgs field $\Phi:M\to \mathfrak{gl}(n,\C)$, is defined as
$$
I_{A,\Phi} f:=u^f\big|_{\p_+ SM}.
$$}
\end{Definition}

It is clear that a general $f\in C^\infty(SM;\C^n)$ cannot be determined by its attenuated thermostat ray transform, since $f$ depends on more variables than $I_{A,\Phi} f$. Moreover, one can easily see that the functions of the following type are always in the kernel of $I_{A,\Phi}$
$$
(\mathbf G_E+A+\Phi)p,\quad p\in C^\infty(SM;\C^n)\text{ such that }p|_{\p(SM)}=0.
$$
However, one often encounters $I_{A,\Phi}$ acting on functions on $SM$ arising from symmetric tensor fields on $M$. We will further consider this particular case.

Let $f_{i_1\cdots i_m}(x)\,dx^{i_1}\otimes\cdots\otimes\,dx^{i_m}$ be a smooth $\C^n$-valued symmetric $m$-tensor field on $M$. Then $f$ induces the corresponding function $f$ on $SM$, defined by
$$
f(x,v):=f_{i_1\cdots i_m}(x)v^{i_1}\cdots v^{i_m},\quad (x,v)\in SM.
$$
We denote by $C^\infty(S^m_{M};\C^n)$ the space of smooth $\C^n$-valued symmetric (covariant) $m$-tensor fields on $M$. We will also use the notations $C^\infty(\Lambda^1_{M};\C^n)$ and $C^\infty(M;\C^n)$ for $C^\infty(S^1_{M};\C^n)$ and $C^\infty(S^0_{M};\C^n)$, respectively. In what follows, we will identify $f\in C^\infty(S^m_{M};\C^n)$ with its corresponding induced function $f\in C^\infty(SM;\C^n)$.

%We denote by $C^\infty(M;\Lambda^1_{\C^n}(M))$ and $C^\infty(M;\C^n)$ the spaces of smooth $\C^n$-valued $1$-forms and functions, respectively, on $M$. Given $f\in C^\infty(M;\C^n)$ and $\alpha\in C^\infty(M;\Lambda^1_{\C^n}(M))$, they induce smooth $\C^n$-valued smooth functions $f$ and $\alpha$ on $SM$ by
%$$
%f(x,v):=f(x)\text{ and }\alpha(x,v):=\alpha_{i}(x)\,v^{i},\qquad (x,v)\in SM.
%$$

For a given $[f,h]\in C^\infty(S^m_{M};\C^n)\times C^\infty(S^{m-1}_{M};\C^n)$, $m\ge 1$, we define
$$
I_{A,\Phi}[f,h]:=I_{A,\Phi}f+I_{A,\Phi}h .
$$
This operator also has non-trivial kernel since $I_{A,\Phi}[\bfG_E p+Ap,\Phi p]=0$ for all $p\in C^\infty(S^{m-1}_M;\C^n)$ such that $p|_{\p M}=0$. In the present paper, we are interested in the question whether these are the only kernels in $I_{A,\Phi}$. Our second main result gives an affirmative answer provided that the thermostat curvature is negative.

\begin{Theorem}\label{th::main 1}
Let $(M,g,E)$ be a strictly convex and non-trapping Gaussian thermostat on a compact oriented surface with boundary such that $K_E<0$. Let also $A:TM\to \mathfrak{gl}(n,\C)$ be a connection and $\Phi:M\to \mathfrak{gl}(n,\C)$ be a Higgs field. Assume that $[f,h]\in C^\infty(S^m_{M};\C^n)\times C^\infty(S^{m-1}_{M};\C^n)$, $m\ge 1$. If $I_{A,\Phi}[f,h]=0$, then $f=\bfG_E p+Ap$ and $h=\Phi p$ for some $p\in C^\infty(S^{m-1}_{M};\C^n)$ with $p|_{\p M}=0$.
\end{Theorem}

The above question was extensively studied in \cite{guillarmou2016negconnections,paternain2018carleman,paternain2012attenuated,paternain2016geodesic,zhou2017generic} for geodesic flows and also for magnetic flows in \cite{ainsworth2013attenuated}. In fact, a result similar to ours was proven in \cite{paternain2018carleman} for geodesic flows in any dimension. For Gaussian thermostats on surfaces, such results were established in \cite{assylbekov2017invariant} in the absence of connections and Higgs fields. When $m=1$, these results hold on Finsler surfaces without any curvature constraints \cite{assylbekov2018general}.

\subsection{Main results on closed surfaces.}
Next, we consider similar problems discussed in Section~\ref{sec::surfaces with boundary} but on closed oriented Riemannian surfaces. Let $(M,g)$ be a closed Riemannian surface and let $E$ be an external field. Consider the canonical projection $\pi:SM\to M$. Since there is no boundary, we need to work with closed thermostat geodesics. Clearly, there should exist sufficiently many closed thermostat geodesics for the question to be sound. This is so, for example, in the case when the thermostat flow is Anosov, (see \cite[Theorem~5.2]{wojtkowski2000w}), the Gaussian thermostat $(M,g,E)$ with $K_E<0$ satisfies Anosov and transitivity properties.

%\subsubsection*{\bf Inverse problems for unitary connections and skew-Hermitian Higgs fields}
Given a pair of unitary connection and skew-Hermitian Higgs fields $(A,\Phi)$, there is a naturally induced $U(n)$-cocycle over the thermostat flow $\phi_t$ on $SM$ associated with the Gaussian thermostat $(M,g,E)$. The \emph{cocyle} corresponding to $(A,\Phi)$ is the unique solution $C_{A,\Phi}:SM\times\R\to U(n)$ of
$$
\big[\p_t +A(\phi_t(x,v))+\Phi(\pi\circ\phi_t(x,v))\big]C_{A,\Phi}(x,v;t)=0,\quad C_{A,\Phi}(x,v;0)=\Id.
$$
\begin{Definition}{\rm
The pair $(A,\Phi)$ is said to be \emph{transparent} if $C_{A,\Phi}(x,v;T)=\Id$ for every time $T$ such that $\phi_T(x,v)=(x,v)$. %If there is $u\in C^\infty(SM;U(n))$ such that
%$$
%C_{A,\Phi}(x,v;t)=u(\phi_t(x,v))u(x,v)^{-1},
%$$
%then the pair $(A,\Phi)$ will be called \emph{cohomologically trivial}. Then $u$ in this case is called a \emph{trivializing function}.
}\end{Definition}

Suppose that $(A,\Phi)$ is transparent and $Q\in C^\infty(M;U(n))$. Then it is not difficult to show that $C(x,v;t):=Q(\phi_t(x,v))^{-1}C_{A,\Phi}(x,v;t)Q(x,v)$ satisfies
$$
\Big[\p_t+(Q^{-1}(d+A)Q)(\phi_t(x,v))+(Q^{-1}\Phi Q)(\pi\circ \phi_t(x,v))\Big]C(x,v;t)=0
$$
and $C(x,v;0)=\Id$. Moreover, one can see that $C(x,v;T)=\Id$ for all $T$ satisfying $\phi_T(x,v)=(x,v)$. Therefore, we can claim that
$$
C_{Q^{-1}(d+A)Q,Q^{-1}\Phi Q}(x,v;t)=Q(\phi_t(x,v))^{-1}C_{A,\Phi}(x,v;t)Q(x,v).
$$
This says that the set of transparent pairs is invariant under the gauge change:
$$
(A,\Phi)\mapsto (Q^{-1}(d+A)Q,Q^{-1}\Phi Q),\qquad Q\in C^\infty(M;U(n)).
$$
We are interested in the problem of classifying all transparent pairs up to gauge equivalence.
%We are interested in the inverse problem of determining $(A,\Phi)$ from the knowledge of $C_{A,\Phi}(x,v;T)$ for all $(x,v)\in SM$ and $T\in\R$ such that $\phi_T(x,v)=(x,v)$. However, similar to the boundary case, there is a natural obstruction for this problem. Indeed, suppose $Q\in C^\infty(M;U(n))$. Then it is not difficult to show that $C(x,v;t):=Q(\phi_t(x,v))^{-1}C_{A,\Phi}(x,v;t)Q(x,v)$ satisfies
%$$
%[\p_t+(Q^{-1}(d+A)Q)(\phi_t(x,v))+(Q^{-1}\Phi Q)(\pi\circ \phi_t(x,v))]C(x,v;t)=0,\,\, C(x,v;0)=\Id.
%$$
%Therefore, we can claim that
%$$
%C_{Q^{-1}(d+A)Q,Q^{-1}\Phi Q}(x,v;t)=C(x,v;t)=Q(\phi_t(x,v))^{-1}C_{A,\Phi}(x,v;t)Q(x,v).
%$$
%Moreover, one can see that $C_{Q^{-1}(d+A)Q,Q^{-1}\Phi Q}(x,v;T)=C_{A,\Phi}(x,v;T)$ for all $T\in\R$ such that $\phi_T(x,v)=(x,v)$. Therefore, from the knowledge of $C_{A,\Phi}(x,v;T)$ for all $(x,v)\in SM$ and $T\in\R$ such that $\phi_T(x,v)=(x,v)$, we can only hope to recover $(A,\Phi)$ up to such an obstruction.

To state our result, we need to introduce a few more notions. Since $A$ is a unitary connection, its curvature $F_A$, defined as $F_A=dA+A\wedge A$, is a smooth $\mathfrak u(n)$-valued $2$-form on $M$. Hence, $\star F_A\in C^\infty(M;\mathfrak u(n))$, where $\star$ is the Hodge star operator on $(M,g)$. Therefore, there exist Lipschitz continuous $\lambda_{\min},\lambda_{\max}:M\to\R$ such that $\lambda_{\min}(x)$ and $\lambda_{\max}(x)$ are smallest and largest, respectively, eigenvalues of $i\star F_A(x)\in\mathfrak u(n)$ at every $x\in M$. As usual, by $\chi(M)$ we denote the Euler characteristic of $M$. Our third main result is as follows.

\begin{Theorem}\label{th::main 4}
Let $(M,g,E)$ be a Gaussian thermostat on a closed oriented surface such that $K_E<0$. Let also $A:TM\to \mathfrak{u}(n)$ be a unitary connection and $\Phi:M\to \mathfrak{gl}(n, \C)$ be a skew-Hermitian Higgs field. %Denote by $\lambda_1\le\dots\le\lambda_n$ the eigenvalues of $i\star F_{A}$ counted with multiplicity. 
Suppose that either
\begin{equation}\label{eqn::condition 1}
2\pi k\chi(M)<\int_M \lambda_{\min}\,d\Vol_g,\quad \int_M\lambda_{\max}\,d\Vol_g<-2\pi k\chi(M)
\end{equation}
or
\begin{equation}\label{eqn::condition 2}
k>\frac{\|i\star F_A\|_{L^\infty(M)}}{\kappa}
\end{equation}
is satisfied for all $k\ge 1$, where $\kappa>0$ is a constant such that $K_E\le -\kappa$. If $C_{A,\Phi}(x,v;T)=\Id$ for all $(x,v)\in SM$ and $T\in\R$ such that $\phi_T(x,v)=(x,v)$, then the pair $(A,\Phi)$ is gauge equivalent to the trivial pair, i.e. $Q^{-1}(d+A)Q=0$ and $\Phi=0$ for some $Q\in C^\infty(M;U(n))$.
\end{Theorem}

For geodesic flows on higher dimensional manifolds, this result was obtained in \cite{guillarmou2016negconnections, paternain2018carleman}. For earlier related results in the case of geodesic flows, see \cite{paternain2008transparent, paternain2011backlund, paternain2012transparent, paternain2013inverse}.

%We study those pairs whose cocycles are invisible along closed thermostat geodesics of $(M,g,E)$.

%\begin{Definition}{\rm
%The pair $(A,\Phi)$ is said to be \emph{transparent} if $C_{A,\Phi}(x,v;T)=\Id$ for every time $T$ such that $\phi_T(x,v)=(x,v)$. %If there is $u\in C^\infty(SM;U(n))$ such that
%$$
%C_{A,\Phi}(x,v;t)=u(\phi_t(x,v))u(x,v)^{-1},
%$$
%then the pair $(A,\Phi)$ will be called \emph{cohomologically trivial}. Then $u$ in this case is called a \emph{trivializing function}.
%}\end{Definition}

%Suppose that $(A,\Phi)$ is transparent and $Q\in C^\infty(M;U(n))$. Then it is not difficult to show that $C(x,v;t):=Q(\phi_t(x,v))^{-1}C_{A,\Phi}(x,v;t)Q(x,v)$ satisfies
%$$
%\Big[\p_t+(Q^{-1}(d+A)Q)(\phi_t(x,v))+(Q^{-1}\Phi Q)(\pi\circ \phi_t(x,v))\Big]C(x,v;t)=0
%$$
%and
%$$
%C(x,v;T)=\Id
%$$
%for all $T$ such that $\phi_T(x,v)=(x,v)$. Therefore, we can claim that
%$$
%C_{Q^{-1}(d+A)Q,Q^{-1}\Phi Q}(x,v;t)=C(x,v;t)=Q(\phi_t(x,v))^{-1}C_{A,\Phi}(x,v;t)Q(x,v).
%$$
%This says that the set of transparent pairs is invariant under the gauge change.

As in the case of surfaces with boundary, this problem is reduced to of the problem to an integral geometry problem. Our final main result is an analog of Theorem~\ref{th::main 1} for closed surfaces.
%\subsubsection*{\bf Inverse problems for transport equations}

\begin{Theorem}\label{th::main 3}
Let $(M,g,E)$ be a Gaussian thermostat on a closed oriented surface such that $K_E<0$. Suppose $A:TM\to \mathfrak{u}(n)$ is a unitary connection for which either \eqref{eqn::condition 1} or \eqref{eqn::condition 2} holds for all $k\ge 1$. Let also $\Phi:M\to \mathfrak{gl}(n,\C)$ be a Higgs field. Assume that $[f,h]\in C^\infty(S^m_{M};\C^n)\times C^\infty(S^{m-1}_{M};\C^n)$, $m\ge 1$. If $u\in C^\infty(SM;\C^n)$ satisfies $(\bfG_E+A+\Phi)u=f+h$, then $u\in C^\infty(S^{m-1}_{M};\C^n)$. Hence, $f=\bfG_E u+Au$ and $h=\Phi u$.
\end{Theorem}

This result was obtained in \cite{guillarmou2016negconnections, paternain2018carleman} for geodesic flows on higher dimensional manifolds. In the absence of connections and Higgs fields, various versions of this result were established in \cite{assylbekov2018general, assylbekov2017invariant, dairbekov2007entropy,dairbekov2007entropy2, jane24injectivity} for Gaussian thermostats.

\subsection{Structure of the paper} The paper is organized as follows. In Section~\ref{section::prelim} we collect certain preliminary facts about geometry and Fourier analysis on $SM$ and Gaussian thermostats. Then we derive a Carleman type estimate for Gaussian thermostats with negative curvature in Section~\ref{section::Carleman est}. Next, in Section~\ref{sec::regularity results} we prove certain regularity results for solutions to the transport equation on non-trapping and strictly convex Gaussian thermostats. Section~\ref{sec::linear problems} contains the proofs of Theorem~\ref{th::main 1} and Theorem~\ref{th::main 3}. In Section~\ref{sec::injectivity of mu+A operators} we prove certain technical results which we used in Section~\ref{sec::linear problems}. Finally, the proofs of Theorem~\ref{th::main 2} and Theorem~\ref{th::main 4} are presented in Section~\ref{sec::nonlinear problems}.

\subsection*{Acknowledgements} YA would like to thank Total E \& P Research \& Technology USA and the members of the Geo-Mathematical Imaging Group at Rice University for financial support.	

%--------------------------------------------------------------------------------------------------------------------------------------------------------------------

\section{Preliminaries}\label{section::prelim}

\subsection{Geometry on $SM$} Since $M$ is assumed to be oriented there is a circle action on the fibres of $SM$ with infinitesimal generator $V$ called the \emph{vertical vector field}. Let X denote the generator of the geodesic flow of $g$. We complete $X,V$ to a global frame of $T(SM)$ by defining the vector field $X_\perp:= [V,X]$, where $[\cdot,\cdot]$ is the Lie bracket for vector fields. Using this frame we can define a Riemannian metric on $SM$ by declaring $\{X, X_{\perp}, V\}$ to be an orthonormal basis and the volume form of this metric will be denoted by $d\Sigma^3$. The vector fields $X,X_\perp,V$ satisfy the following structural equations
\begin{equation}\label{eqn::structural commutation equations}
X=[V, X_{\perp}],\quad X_{\perp}=[X,V],\quad [X, X_{\perp}]=-KV,
\end{equation}
where $K$ is the Gaussian curvature of the surface.

\subsection{Fourier analysis on $SM$}
For any two functions $u,v:SM\to\C^n$ define the $L^2$ inner product as
$$
(u,v):=\int_{SM} \<u,v\>_{\C^n}\,d\Sigma^3
$$
with the corresponding norm denoted by $\|\cdot\|$. The space $L^2(SM;\C^n)$ decomposes orthogonally as a direct sum $L^2(SM;\C^n)=\bigoplus_{k\in \mathbb Z}H_k(SM;\C^n)$ where $H_k(SM;\C^n)$ is the eigenspace of $V$ corresponding to the eigenvalue $ik$. A function $u\in L^2(SM;\C^n)$ has a Fourier series expansion
$$
u=\sum^{\infty}_{k=-\infty} u_k,\qquad u_k\in H_k(SM;\C^n).
$$
Then $\|u\|^2=\sum_{k=-\infty}^\infty\|u_k\|^2$. We denote the subspace
$$
\Omega_k(SM;\C^n):=H_k(SM;\C^n)\cap C^{\infty}(SM;\C^n).
$$
We say that $u\in C^\infty(SM;\C^n)$ is of \emph{degree} $m$ if $u_k=0$ for all $|k|\ge m+1$. If $f\in C^\infty(S^m_M;\C^n)$, then the corresponding $f\in C^\infty(SM;\C^n)$ is of degree $m$; see \cite[Section~2]{paternain2013tensor} for more details. Moreover, if $m$ is odd/even then $f_k=0$ for all even/odd $k\in\Z$. In particular, for any $\alpha\in C^\infty(\Lambda^1_M;\C^n)$ its induced function $\alpha$ on $SM$ can be written as $\alpha=\alpha_{-1}+\alpha_1$ with $\alpha_{\pm 1}\in \Omega_{\pm 1}(SM;\C^n)$.

\subsection{Generating vector field of a Gaussian thermostat} Let $(M,g,E)$ be a Gaussian thermostat with $\dim(M)=2$. Then according to \cite{dairbekov2007entropy}, the generating vector field $\bfG_E$ of the thermostat flow $\phi_t$ can be written as
$$
\bfG_E=X+\lambda V.
$$
For the adjoints of $X_\perp$, $V$ and $\bfG_E$, w.r.t. $(\cdot,\cdot)$, we have
$$
X_\perp^*=-X_\perp,\quad V^*=-V,\quad \bfG_E^*=-(\bfG_E+V(\lambda));
$$
see \cite{assylbekov2017invariant,dairbekov2007entropy}.
 
\subsection{Guillemin-Kazhdan type operators} Consider the following first order differential operators introduced by Guillemin and Kazhdan~\cite{guillemin1980some}
\begin{equation}\label{def::eta operators}
\eta_-=\frac{1}{2}(X-iX_\perp),\qquad \eta_+=\frac{1}{2}(X+iX_\perp).
\end{equation}
Then $X$ can be expressed as $X=\eta_-+\eta_+$. It was shown that $\eta_\pm:\Omega_k(SM;\C^n)\to\Omega_{k\pm1}(SM;\C^n)$, $k\in\Z$. Moreover, these operators are elliptic and satisfy the following commutation relation
\begin{equation}\label{eq::[eta_+,eta_-]}
[\eta_+,\eta_-]=\frac{i}{2}KV.
\end{equation}
It is also important to note that $\eta_-^*=-\eta_+$ and $\eta_+^*=-\eta_-$; see \cite{guillemin1980some}.

For the Gaussian thermostat $(M,g,E)$, we consider the modified versions of $\eta_+$ and $\eta_-$ defined as
$$
\mu_-:=\eta_-+\lambda_{-}V,\quad \mu_+:=\eta_++\lambda_+ V,
$$
where $\lambda=\lambda_{-}+\lambda_{+}$ with $\lambda_\pm\in\Omega_{\pm 1}(SM;\C)$. Then it is not difficult to see that $\bfG_E=\mu_-+\mu_+$.

Using $V^*=-V$ and $\overline{\lambda}_-=\lambda_+$, one can see that
\begin{equation}\label{eqn::adjoints of mu operators}
(\mu_-)^*=-\mu_+-i\lambda_+,\qquad (\mu_+)^*=-\mu_-+i\lambda_-.
\end{equation}
We also have the following analog of \eqref{eq::[eta_+,eta_-]}.

\begin{Lemma}\label{lem::commutation between mu operators}
If $u\in C^\infty(SM;\C^n)$, then
$$
[\mu_+,\mu_-]u=\frac{i}{2}K_EVu-i\lambda_-\mu_+u-i\lambda_+\mu_- u.
$$
\end{Lemma}

\begin{proof}
It is enough to give the proof in the case $u\in \Omega_k(SM;\C^n)$. Using \eqref{eq::[eta_+,eta_-]} and the fact that $Vu=iku$,
\begin{align*}
[\mu_+,\mu_-]u&=[\eta_+,\eta_-]u+[\eta_+,\lambda_- V]u+[\lambda_+V,\eta_-]u+[\lambda_+V,\lambda_-V]u\\
&=\frac{i}{2}KVu+\eta_+(ik\lambda_- u)-i(k+1)\lambda_-\eta_+ u+i(k-1)\lambda_+\eta_- u\\
&\qquad-\eta_-(ik\lambda_+ u)-k(k-1)\lambda_+\lambda_- u+k(k-1)\lambda_+\lambda_- u\\
&=\frac{i}{2}KVu+ik(\eta_+\lambda_--\eta_-\lambda_+)u-i\lambda_-\eta_+u-i\lambda_+\eta_- u+2k\lambda_+\lambda_- u\\
&=\frac{i}{2} KVu+(\eta_+\lambda_--\eta_-\lambda_+)Vu-i\lambda_-\mu_+u-i\lambda_+\mu_- u.
\end{align*}
Now, let $\theta$ be a $1$-form on $M$ dual to $E$. Then $\lambda=V\theta$, and hence $\lambda_-=-i\theta_{-1}$ and $\lambda_+=i\theta_1$. Using these, a straightforward calculation shows that
$$
\eta_+\lambda_--\eta_-\lambda_+=-\frac{i}{2}(X\theta-X_\perp V\theta).
$$
However, $X\theta-X_\perp V\theta=\div_g E$; see the proof of \cite[Lemma~5.2]{paternain2006regularity}. Thus, we have
$$
(\eta_+\lambda_--\eta_-\lambda_+)=-\frac{i}{2}\div_g E,
$$
which completes the proof.
\end{proof}

\begin{Lemma}\label{lem::commutation formula for mu+A operators}
If $A:TM\to \mathfrak{gl}(n,\C)$ is a connection, then
$$
[\mu_{+}+A_{+},\mu_{-}+A_{-}]u=\frac{i}{2}K_EVu+\frac{i}{2}\star F_A u-i\lambda_-(\mu_{+}+A_{+})u-i\lambda_+(\mu_{-}+A_{-})u
$$
for $u\in \Omega_k(SM;\C^n)$.
\end{Lemma}
\begin{proof}
Using Lemma~\ref{lem::commutation between mu operators},
\begin{align*}
&[\mu_{+}+A_{+},\mu_{-}+A_{-}]u=[\mu_{+},\mu_{-}]u+[A_{+},\mu_{-}]u+[\mu_{+},A_{-}]u+[A_{+},A_{-}]u\\
&=\frac{i}{2}K_EVu-i\lambda_{-}(\mu_{+}+A_{+})u-i\lambda_{+}(\mu_{-}+A_{-})u+(\eta_{+}A_{-}-\eta_{-}A_{+})u+[A_{+},A_{-}]u.
\end{align*}
A straightforward calculation shows that
$$
\eta_{+}A_{-}-\eta_{-}A_{+}=\frac{i}{2}(X_\perp A+XVA),\quad [A_{+},A_{-}]=-\frac{i}{2}[VA,A].
$$
Hence, by (6) in \cite{paternain2012attenuated}, we have
$$
\eta_{+}A_{-}-\eta_{-}A_{+}+[A_{+},A_{-}]=\frac{i}{2}\star F_A,
$$
Using this, we get the desired identity.
\end{proof}

%--------------------------------------------------------------------------------------------------------------------------------------------------------------------

\section{Carleman estimate for Gaussian thermostats with negative curvature}\label{section::Carleman est}
In the present section, we prove a Carleman estimate for Gaussian thermostats with negative curvature, which will be used in Section~\ref{sec::linear problems} to prove our main results.
\begin{Theorem}\label{th::Carleman estimate}
Let $(M,g,E)$ be a Gaussian thermostat on a compact oriented surface and $m\ge 1$ be an integer. Assume that $K_E\le-\kappa$ for some $\kappa>0$ constant. Then for any $s>0$, we have
$$
\sum_{k=m}^\infty k^{2s+1} \Big(\|u_k\|^2+\|u_{-k}\|^2\Big)\le \frac{1}{\kappa s}\sum_{k=m+1}^\infty k^{2s+1}\Big(\|(\bfG_E u)_k\|^2+\|(\bfG_E u)_{-k}\|^2\Big),
$$
for all $u\in C^\infty(SM;\C^n)$, with $u|_{\p(SM)}=0$ in the case $\p M\neq \varnothing$.
\end{Theorem}

%\subsection{Guillemin-Kazhdan type energy identity}
To prove Theorem~\ref{th::Carleman estimate}, we follow the arguments in~\cite[Section~6]{paternain2018carleman}. Our starting point will be the following analog of Guillemin-Kazhdan energy identity~\cite{guillemin1980some}.

\begin{Proposition}\label{prop::GK energy identity for mu operators}
Let $(M,g,E)$ be a Gaussian thermostat on a compact oriented surface. Then for any $u\in \Omega_k(SM;\C^n)$, $k\in\Z$, with $u|_{\p(SM)}=0$ in the case $\p M\neq \varnothing$, we have
$$
\|\mu_+ u\|^2=\|\mu_- u\|^2-\frac{k}{2}(K_E u,u).
$$
\end{Proposition}

\begin{Remark}{\rm
In fact, this identity was derived earlier in \cite{assylbekov2017invariant} (see the the proof of Lemma~6.5 therein) via the Pestov identity. Here we give an alternative proof without the involvement of the Pestov identity. It also follows from more general Proposition~\ref{prop::weighted GK energy identity for mu+A operators} in Section~\ref{sec::injectivity of mu+A operators}.
}\end{Remark}

\begin{proof}
Using \eqref{eqn::adjoints of mu operators} and then Lemma~\ref{lem::commutation between mu operators},
\begin{align*}
\|\mu_+ u\|^2&=(\mu_+u,\mu_+u)=(\mu_+^*\mu_+u,u)=(-\mu_-\mu_+u,u)+(i\lambda_-\mu_+u,u)\\
&=-(\mu_+\mu_-u,u)+\frac{i}{2}(K_EVu,u)-(i\lambda_+\mu_-u,u)\\
&=(\mu_-^*\mu_-u,u)+\frac{i}{2}(K_EVu,u)=\|\mu_-u\|^2-\frac{k}{2}(K_Eu,u).
\end{align*}
The proof is thus complete.
\end{proof}

The proposition above now yields the following result.

\begin{proof}[Proof of Theorem~\ref{th::Carleman estimate}]
Using the hypothesis $K_E\le-\kappa$, for $k\ge m$ integer, we show
\begin{align*}
\|\mu_- u_k\|^2+\|\mu_+ &u_{-k}\|^2+\frac{k\kappa}{2}\Big(\|u_k\|^2+\|u_{-k}\|^2\Big)\\
&\quad\le \|\mu_-u_k\|^2-\frac{k}{2}(K_Eu_k,u_k)+\|\mu_+u_{-k}\|^2+\frac{(-k)}{2}(K_E u_{-k},u_{-k})\\
&\quad= \|\mu_+u_k\|^2+\|\mu_-u_{-k}\|^2,
\end{align*}
where at the last step we used Proposition~\ref{prop::GK energy identity for mu operators} to $u_k$ and $u_{-k}$. Multiplying this estimate by $k^{2s}$ for some $s>0$, we get
%For $s>0$, write $\gamma_k=k^{2s}$. Then by Corollary~\ref{cor::GK energy estimate},
\begin{align*}
k^{2s}\Big(\|\mu_- u_k\|^2+\|\mu_+ u_{-k}\|^2\Big)&+\frac{\kappa}{2}k^{2s+1}\Big(\|u_k\|^2+\|u_{-k}\|^2\Big)\\
&\le k^{2s}\Big(\|\mu_+ u_k\|^2+\|\mu_- u_{-k}\|^2\Big)\\
&\le k^{2s}\Big(1+\frac{1}{\varepsilon_k}\Big)\Big(\|(\bfG_Eu)_{k+1}\|^2+\|(\bfG_Eu)_{-k-1}\|^2\Big)\\
&\quad+k^{2s}(1+\varepsilon_k)\Big(\|\mu_-u_{k+2}\|^2+\|\mu_+u_{-k-2}\|^2\Big),
\end{align*}
where $\{\varepsilon_k\}_{k=m}^\infty$ is a sequence of positive numbers to be chosen later. Summing these estimates over $k$ from $m$ to $N$, we obtain
\begin{multline*}
\sum_{k=m}^N k^{2s}\Big(\|\mu_- u_k\|^2+\|\mu_+ u_{-k}\|^2\Big)+\sum_{k=m}^N \frac{\kappa}{2}k^{2s+1}\Big(\|u_k\|^2+\|u_{-k}\|^2\Big)\\
\le \sum_{k=m+1}^{N+1} (k-1)^{2s}\Big(1+\frac{1}{\varepsilon_{k-1}}\Big)\Big(\|(\bfG_Eu)_{k}\|^2+\|(\bfG_Eu)_{-k}\|^2\Big)\\
+\sum_{k=m+2}^{N+2} (k-2)^{2s}(1+\varepsilon_{k-2})\Big(\|\mu_-u_{k}\|^2+\|\mu_+u_{-k}\|^2\Big).
\end{multline*}
Note here we are only considering the first N terms of the series, with the intent
of achieving convergence later. Now, we make an appropriate choice of $\varepsilon_k$, i.e.
$$
\varepsilon_{k-2}=\frac{k^{2s}}{(k-2)^{2s}}-1>0,\quad k\ge m+2,
$$%According to hypothesis \eqref{eqn::condition for gamma}
which yields
$$
(k-2)^{2s}(1+\varepsilon_{k-2})=k^{2s},\quad\text{for all}\quad k\ge m+2,
$$
and
$$
(k-1)^{2s}\Big(1+\frac{1}{\varepsilon_{k-1}}\Big)=\frac{(k^2-1)^{2s}}{(k+1)^{2s}-(k-1)^{2s}}\le \frac{k^{2s+1}}{2s},
$$
where in the last step we used \cite[Lemma~6.8]{paternain2018carleman}. Inserting these into the above estimate, we get
\begin{multline*}
\sum_{k=m}^{m+1} k^{2s}\Big(\|\mu_- u_k\|^2+\|\mu_+ u_{-k}\|^2\Big)+\sum_{k=m}^N \frac{\kappa}{2}k^{2s+1}\Big(\|u_k\|^2+\|u_{-k}\|^2\Big)\le\\
 \sum_{k=m+1}^{N+1} \frac{k^{2s+1}}{2s} \Big(\|(\bfG_Eu)_{k}\|^2+\|(\bfG_Eu)_{-k}\|^2\Big)+\sum_{k=N+1}^{N+2} k^{2s}\Big(\|\mu_-u_{k}\|^2+\|\mu_+u_{-k}\|^2\Big).
\end{multline*}
Now we take the limit as $N\to\infty$. Since the weights $k^{2s+1}$ grow at most polynomially for each fixed $s>0$, the last term goes to zero. Therefore, in particular, we have
$$
\sum_{k=m}^\infty k^{2s+1} \Big(\|u_k\|^2+\|u_{-k}\|^2\Big)\le \frac{1}{\kappa s}\sum_{k=m+1}^\infty k^{2s+1}\Big(\|(\bfG_E u)_k\|^2+\|(\bfG_E u)_{-k}\|^2\Big)
$$
as desired. This completes the proof.
\end{proof}

%\begin{proof}[Proof of Theorem~\ref{th::Carleman estimate}]
%Take $\gamma_k=k^{2s}$, $s\in\R$, which clearly satisfies the hypothesis \eqref{eqn::condition for gamma} in Proposition~\ref{prop::pre-Carleman estimate}. Therefore, in particular, we have
%\begin{multline*}
%\sum_{k=m}^\infty k^{2s+1} \Big(\|u_k\|^2+\|u_{-k}\|^2\Big)\\
%\le \frac{2}{\kappa}\sum_{k=m+1}^\infty\frac{k^{4s}}{(k+1)^{2s}-(k-1)^{2s}}\Big(\|(\bfG_E u)_k\|^2+\|(\bfG_E u)_{-k}\|^2\Big)
%\end{multline*}
%for any $u\in C^\infty_F(SM;\C^n)$ satisfying $u|_{\p(SM)}=0$ if $\p M\neq\varnothing$. Now, by \cite[Lemma~6.8]{paternain2018carleman}, for all $s>0$ we obtain the desired estimate
%$$
%\sum_{k=m}^\infty k^{2s+1} \Big(\|u_k\|^2+\|u_{-k}\|^2\Big)\le \frac{1}{\kappa s}\sum_{k=m+1}^\infty k^{2s+1}\Big(\|(\bfG_E u)_k\|^2+\|(\bfG_E u)_{-k}\|^2\Big).
%$$
%Since the weights $k^{2s+1}$ grow at most polynomially for each fixed $s>0$, this estimate still holds true for all $u\in C^\infty(SM;\C^n)$ with $u|_{\p(SM)}=0$ when $\p M\neq\varnothing$.
%\end{proof}

%--------------------------------------------------------------------------------------------------------------------------------------------------------------------

\section{Regularity results for the transport equations}\label{sec::regularity results}
Let $(M,g,E)$ be a non-trapping and strictly convex Gaussian thermostat on a compact manifold with boundary. Suppose that $A:TM\to \mathfrak{gl}(n,\C)$ is a connection and $\Phi:M\to \mathfrak{gl}(n,\C)$ is a Higgs field on the trivial bundle $M\times \C^n$. 

\subsection{Scattering relation}
Recall that $\tau(x,v)$, for $(x,v)\in SM$, is the first non-negative time when the thermostat geodesic $\gamma_{x,v}$ exits $M$. Since $(M,g,E)$ is non-trapping and strictly convex, $\tau$ is continuous on $SM$ and smooth on $SM\setminus S(\p M)$.
\begin{Lemma}\label{lem::tau is smooth on influx bundle}
For a non-trapping and strictly convex Gaussian thermostat $(M,g,E)$ on a compact manifold with boundary, the restricted function $\tau|_{\p_+ SM}$ is smooth on $\p_+ SM$.
\end{Lemma}
\begin{proof}
Consider a smooth $\rho:M\to [0,\infty)$ such that $\rho^{-1}(0)=\p M$ and $|\nabla \rho|_{g}\equiv 1$ near $\p M$. Clearly, $\nabla\rho=\nu$. Write $h(x,v;t)=\rho(\gamma_{x,v}(t))$ for $(x,v)\in\p_+ SM$. Then
$$
h(x,v;0)=0,\quad \p_t h(x,v;0)=\<\nu(x),v\>_{g(x)},
$$
$$
\p_t^2 h(x,v;0)=\Hess_x\rho(v,v)+\<\nu(x),E(x)\>_{g(x)}.
$$
Therefore, we have
$$
h(x,v;t)=\<\nu(x),v\>_{g(x)}t+\frac{1}{2}(\Hess_x\rho(v,v)+\<\nu(x),E(x)\>_{g(x)})t^2+R(x,v;t)t^3
$$
for a smooth function $R$. Since $h(x,v;\tau(x,v))=0$, one can see that $T=\tau(x,v)$ solves the equation
$$
\<\nu(x),v\>_{g(x)}+\frac{1}{2}(\Hess_x\rho(v,v)+\<\nu(x),E(x)\>_{g(x)})T+R(x,v;T)T^2=0.
$$
Let us denote the left side of the above equation by $F(x,v;T)$. Then for $(x,v)\in S(\p M)$,
$$
2\p_TF(x,v;0)=(-\Lambda(x,v)+\<\nu(x),E(x)\>_{g(x)}).
$$
Here we used the fact that $\Hess_x\rho(v,v)=-\Lambda(x,v)$ for $(x,v)\in S(\p M)$. By strict convexity, we have $\p_TF(x,v;0)<0$ for $(x,v)\in S(\p M)$. Therefore, the Implicit Function Theorem implies the smoothness of $\tau$ in a neighborhood of $S(\p M)$. This completes the proof, since $\tau$ is smooth on $SM\setminus S(\p M)$.
\end{proof}

The \emph{scattering relation} is the map $\mathcal S:\p_+SM\to\p_-SM$ defined as
$$
\mathcal S(x,v)=\big(\gamma_{x,v}(\tau(x,v)),\dot\gamma_{x,v}(\tau(x,v))\big),\quad (x,v)\in \p_+SM.
$$
In other words, $\mathcal S$ maps the point and direction of entrance of a thermostat geodesic to the point and direction of the exit. By Lemma~\ref{lem::tau is smooth on influx bundle}, the scattering relation $\mathcal S$ is a diffeomorphism between $\p_+ SM$ and $\p_- SM$.

\subsection{Functions constant along the thermostat flow}
For $w\in C^\infty(\p_+SM;\C^n)$, the problem
$$
\mathbf G_E w_{\psi}=0\quad\text{in}\quad SM,\qquad w_{\psi}\big|_{\p_+ SM}=w
$$
has the unique solution $w_{\psi}:SM\to\C^n$. In other words, $w_{\psi}$ is the function that is constant along the orbits of the thermostat flow and equals $w$ on $\p_+ SM$.

Define the space
$$
C^\infty_\alpha(\p_+SM;\C^n):=\{w\in C^\infty(\p_+SM;\C^n):w_\psi\in C^\infty(SM;\C^n)\}.
$$
Consider the operator $\mathcal A:C(\p_+SM;\C^n)\to C(\p(SM);\C^n)$ given by
$$
\mathcal Aw(x,v):=\begin{cases}
w(x,v), &(x,v)\in\p_+SM,\\
(w\circ\mathcal S^{-1})(x,v), &(x,v)\in\p_- SM.
\end{cases}
$$
Clearly, $w_\psi|_{\p(SM)}=\mathcal Aw$. According to Theorem~\ref{th::characterization of smooth functions constant along the flow}, the space $C^\infty_\alpha(\p_+SM;\C^n)$ can be characterized as follows:
\begin{equation}\label{eqn::characterization of smooth functions constant along the flow}
C^\infty_\alpha(\p_+SM;\C^n):=\{w\in C^\infty(\p_+SM;\C^n):\mathcal Aw\in C^\infty(\p(SM);\C^n)\}.
\end{equation}

\subsection{Homogeneous equations}
For a given $w\in C^\infty(\p_+SM;\C^n)$, consider the unique solution $w^\sharp:SM\to \C^n$ for the transport equation
$$
(\mathbf G_E+A+\Phi)w^\sharp=0\quad\text{in}\quad SM,\qquad w^\sharp\big|_{\p_+ SM}=w.
$$
Define the space
$$
\mathcal S^\infty_{A,\Phi}(\p_+SM;\C^n):=\{w\in C^\infty(\p_+SM;\C^n): w^\sharp\in C^\infty(SM;\C^n)\}.
$$
One can see that
$$
w^\sharp(x,v)=U_{A,\Phi}(x,v)w_{\psi}(x,v),\qquad (x,v)\in SM,
$$
where $U_{A,\Phi}$ introduced in Section~\ref{sec::surfaces with boundary}. Introduce the operator
$$
\mathcal Q:C(\p_+SM;\C^n)\to C(\p(SM);\C^n) 
$$
defined, using the scattering relation $\mathcal S$ and the scattering data $C_{A,\Phi}$, as
$$
\mathcal Qw(x,v)=\begin{cases}
w(x,v), &(x,v)\in\p_+SM,\\
C_{A,\Phi}(x,v)(w\circ\mathcal S^{-1})(x,v), &(x,v)\in\p_- SM.
\end{cases}
$$
Then note that $w^\sharp|_{\p(SM)}=\mathcal Qw$. Now, we characterize the space $\mathcal S^\infty_{A,\Phi}(\p_+SM;\C^n)$ in terms of the operator $\mathcal Q$.
\begin{Lemma}\label{lem::characterization of smooth functions constant for X+a}
We have the following characterization
$$
\mathcal S^\infty_{A,\Phi}(\p_+SM;\C^n):=\{w\in C^\infty(\p_+SM;\C^n): \mathcal Qw\in C^\infty(\p(SM);\C^n)\}.
$$
\end{Lemma}
\begin{proof}
Consider  a closed manifold $N$ containing $M$. We smoothly extend the metric $g$ and the vector field $E$ to $N$ and denote extensions by the same notations. Next, embed $M$ into the interior of a compact manifold $M_1\subset N$ with boundary. Choose $M_1$ to be sufficiently close to $M$ so that $(M_1,g,E)$ is also non-trapping and strictly convex.  We also extend $A$ and $\Phi$ smoothly to $N$.

Consider the unique solution $R:SM_1\to GL(n,\C)$ to the transport equation
$$
(\mathbf G_E+A+\Phi)R=0\quad\text{in}\quad SM_1,\qquad R\big|_{\p_+ SM_1}=\Id.
$$
Then the restriction of $R$ to $SM$, still denoted by $R$, is in $C^\infty(SM;GL(n,\C))$ and solves $(\mathbf G_E+A+\Phi)R=0$ in $SM$. Set $r:=R^{-1}|_{\p_+ SM}$. Then we can express $w^\sharp$ and $\mathcal Qw$ as
$$
w^\sharp=R(rw)_{\psi}\quad\text{and}\quad \mathcal Qw(x,v)=\begin{cases}
R(x,v)r(x,v)w(x,v),&\text{if}\quad (x,v)\in\p_+ SM,\\
R(x,v)((rw)\circ\mathcal S^{-1})(x,v),&\text{if}\quad (x,v)\in\p_- SM.
\end{cases}
$$
From this it is obvious that $\mathcal Qw\in C^\infty(\p(SM);\C^n)$ if $w^\sharp\in C^\infty(SM;\C^n)$. Now suppose $\mathcal Qw\in C^\infty(\p(SM);\C^n)$. Since $R\in C^\infty(SM;GL(n,\C))$ and $\mathcal A(rw)=R^{-1}|_{\p(SM)}\mathcal Qw$, this, together with \eqref{eqn::characterization of smooth functions constant along the flow}, implies that $(rw)_\psi\in C^\infty(SM;\C^n)$. Using the fact $R\in C^\infty(SM;GL(n,\C))$ once again, we can claim $w^\sharp\in C^\infty(SM;\C^n)$, finishing the proof.
\end{proof}

\subsection{Non-homogeneous equations}
Given $f\in C^\infty(SM;\C^n)$, consider the unique solution $u^f:SM\to \C^n$ for the transport equation
$$
(\bfG_E+A+\Phi)u=-f\quad\text{in}\quad SM,\qquad u\big|_{\p_- SM}=0.
$$
Note that $U_{A,\Phi}^{-1}$ solves $\bfG_EU_{A,\Phi}^{-1}-U_{A,\Phi}^{-1}(A+\Phi)=0$. Therefore, $\bfG_E(U_{A,\Phi}^{-1}u^f)=-U_{A,\Phi}^{-1}f$. Integrating along $\gamma_{x,v}:[0,\tau(x,v)]\to M$, for $(x,v)\in\p_+SM$, we obtain the following integral expression
$$
u^f(x,v)=\int_0^{\tau(x,v)}U_{A,\Phi}^{-1}(\gamma_{x,v}(t),\dot\gamma_{x,v}(t))f(\gamma_{x,v}(t),\dot\gamma_{x,v}(t))\,dt,\quad (x,v)\in\p_+ SM.
$$
Since $\tau$ is a non-smooth function on $SM$ in general, the function $u^f$ also may fail to be smooth on $SM$. But we can and shall show that $u^f\in C^\infty(SM;\C^n)$ if $I_{A,\Phi}f=0$.
\begin{Proposition}\label{prop::u^f is smooth if If=0}
Let $(M,g,E)$ be a non-trapping and strictly convex Gaussian thermostat on a compact manifold with boundary. Suppose that $A:TM\to \mathfrak{gl}(n,\C)$ is a connection and $\Phi:M\to \mathfrak{gl}(n,\C)$ is a Higgs field on the trivial bundle $M\times \C^n$. If $I_{A,\Phi}f=0$ for $f\in C^\infty(SM;\C^n)$, then $u^f\in C^\infty(SM;\C^n)$.
\end{Proposition}
\begin{proof}
Let $N$ and $M_1$ be as in the proof of Lemma~\ref{lem::characterization of smooth functions constant for X+a}. We consider the same extensions of $g$, $E$, $A$ and $\Phi$ to $N$. We also extend $f$ smoothly to $SN$ preserving the former notation for extension.

Consider the unique solution $a:SM_1\to\C^n$ for the transport equation
$$
(\bfG_E+A+\Phi)a=-f\quad\text{in}\quad SM_1,\qquad a\big|_{\p_- SM_1}=0.
$$
Then $a|_{SM}$, which we keep denoting by $a$, is in $C^\infty(SM;\C^n)$ and solves $(\mathbf G_E+A+\Phi)a=-f$ in $SM$. Then $w:=u^f-a$ satisfies $(\mathbf G_E+A+\Phi)w=0$ in $SM$. Since $u^f|_{\p_- SM}=0$ and $u^f|_{\p_+ SM}=I_{A,\Phi}f=0$, we have
$$
\mathcal Q(w|_{\p_+ SM})=w|_{\p(SM)}=-a|_{\p(SM)}
$$
which is in $C^\infty(\p(SM);\C^n)$. Then $w\in C^\infty(SM;\C^n)$ by Lemma~\ref{lem::characterization of smooth functions constant for X+a}, and hence, $u^f\in C^\infty(SM;\C^n)$. The proof is complete.
\end{proof}

%--------------------------------------------------------------------------------------------------------------------------------------------------------------------

\section{Injectivity results for the linear problems}\label{sec::linear problems}
In the present section, we prove injectivity results for the linear problems stated in the introduction, namely Theorem~\ref{th::main 1} and Theorem~\ref{th::main 3}. To that end, we first need the following theorem.
\begin{Theorem}\label{th::finite degree}
Let $(M,g,E)$ be a Gaussian thermostat on a compact oriented surface such that $K_E<0$. Let also $A:TM\to\mathfrak{gl}(n,\C)$ be a connection and $\Phi:M\to \mathfrak{gl}(n,\C)$ be a Higgs field. Assume that $f\in C^\infty(SM;\C^n)$ has finite degree. If $u\in C^\infty(SM;\C^n)$, with $u|_{\p(SM)}=0$ in the case $\p M\neq \varnothing$, satisfies
$$
(\bfG_E+A+\Phi)u=f\quad\text{in}\quad SM,
$$
then $u$ also has finite degree.
\end{Theorem}
\begin{proof}
Suppose that $f_k=0$ for all $|k|\ge m'$ for some $m'\ge 0$ and take some $m\ge m'$. Then, writing $A=A_-+A_+$ with $A_\pm\in\Omega_{\pm 1}(SM;\C^n)$,
$$
(\bfG_E u)_k=-A_{-}u_{k+1}-A_{+}u_{k-1}-\Phi u_k\quad\text{for all}\quad |k|\ge m.
$$
Therefore, there is $R>0$ constant such that
$$
\|(\bfG_E u)_k\|^2\le R\Big(\|u_{k+1}\|^2+\|u_{k-1}\|^2+\|u_{k}\|^2\Big)\quad\text{for all}\quad |k|\ge m.
$$
Using these in the estimate from Theorem~\ref{th::Carleman estimate}, we obtain
\begin{align*}
&\sum_{k=m}^\infty k^{2s+1} \Big(\|u_k\|^2+\|u_{-k}\|^2\Big)\\
&\le \frac{R}{\kappa s}\sum_{k=m+1}^\infty k^{2s+1}\Big(\|u_{k+1}\|^2+\|u_{k-1}\|^2+\|u_{k}\|^2+\|u_{-k+1}\|^2+\|u_{-k-1}\|^2+\|u_{-k}\|^2\Big)\\
&\le \frac{C}{\kappa s}\sum_{k=m}^\infty (k+1)^{2s+1}\Big(\|u_{k}\|^2+\|u_{-k}\|^2\Big).
\end{align*}
Let us take $s>0$ such that $m\ge 2s+1$. Then
$$
(k+1)^{2s+1}=\Big(1+\frac{1}{k}\Big)^{2s+1}k^{2s+1}\le \Big(1+\frac{1}{k}\Big)^{k}k^{2s+1}\le e k^{2s+1}\quad\text{for all}\quad k\ge m.
$$
Hence, for $m\ge 2s+1$, we have
$$
\sum_{k=m}^\infty k^{2s+1} \Big(\|u_k\|^2+\|u_{-k}\|^2\Big)\le \frac{eC}{\kappa s}\sum_{k=m}^\infty k^{2s+1}\Big(\|u_{k}\|^2+\|u_{-k}\|^2\Big)
$$
Now, we fix $s>eC/\kappa$ and $m=\max(2s+1,m')$. Then
$$
\Big(1-\frac{eC}{\kappa s}\Big)\sum_{k=m}^\infty k^{2s+1} \Big(\|u_k\|^2+\|u_{-k}\|^2\Big)\le 0.
$$
Since $(1-eC/\kappa s)>0$, this allows us to conclude $u_k=0$ for all $|k|\ge m$.
\end{proof}

\subsection{Surfaces with boundary}
The following result is a restatement of Theorem~\ref{th::main 1}.
\begin{Theorem}\label{th::injectivity for tensors}
Let $(M,g,E)$ be a strictly convex and non-trapping Gaussian thermostat on a compact oriented surface with boundary such that $K_E<0$. Let also $A:TM\to \mathfrak{gl}(n,\C)$ be a connection and $\Phi:M\to \mathfrak{gl}(n,\C)$ be a Higgs field. For $m\ge 0$ integer, assume that $f\in C^\infty(SM;\C^n)$ with $f_k=0$ for all $|k|\ge m+1$. If $I_{A,\Phi}f=0$, then $f=(\bfG_E+A+\Phi)u$ for some $u\in C^\infty(SM;\C^n)$ with $u|_{\p(SM)}=0$ such that $u_k=0$ for all $|k|\ge m$.
\end{Theorem}
\begin{proof}
By the results of Section~\ref{sec::regularity results}, there is $u\in C^\infty(SM;\C^n)$ with $u|_{\p(SM)}=0$ such that
\begin{equation}\label{eqn::transport equation}
(\bfG_E+A+\Phi)u=f\quad\text{in}\quad SM.
\end{equation}
Then $u$ is of finite degree by Theorem~\ref{th::finite degree}. Hence, there is $l>0$ such that $u_k=0$ for all $|k|\ge l$. Now, our goal is to show that $l\le m$. Suppose this is not the case, i.e. $l>m$. Then from \eqref{eqn::transport equation}, we get $(\mu_++A_{+})u_{l-1}=0$ and $(\mu_-+A_{-})u_{-l+1}=0$. Since $u_{-l+1}|_{\p(SM)}=u_{l-1}|_{\p(SM)}=0$, Theorem~\ref{th::injectivity of mu+A operators in boundary case} of the next section implies that $u_{-l+1}=u_{l-1}=0$. Continuing this process, we show that $u_k=0$ for all $|k|\ge m$. This concludes the proof.
\end{proof}

\subsection{Closed surfaces}
Similarly, we restate Theorem~\ref{th::main 3} as follows.
\begin{Theorem}\label{th::injectivity for tensors}
Let $(M,g,E)$ be a  Gaussian thermostat on a closed oriented surface such that $K_E\le -\kappa$ for some constant $\kappa>0$. Suppose $A:TM\to \mathfrak{u}(n)$ is a unitary connection for which either
\eqref{eqn::condition 1}  or \eqref{eqn::condition 2} is satisfied for all $k\ge 1$. Let also $\Phi:M\to \mathfrak{gl}(n,\C)$ be a Higgs field. For $m\ge 0$ integer, assume that $f\in C^\infty(SM;\C^n)$ with $f_k=0$ for all $|k|\ge m+1$. If $u\in C^\infty(SM;\C^n)$ satisfies $(\bfG_E+A+\Phi)u=f$, then $u_k=0$ for all $|k|\ge m$.
\end{Theorem}
\begin{proof}
Theorem~\ref{th::finite degree} allows us to claim that $u$ is of finite degree. Then, parallel to arguments in Theorem~\ref{th::injectivity for tensors}, the proof is reduced to the injectivity of $\mu_{\pm}+A_{\pm}$ on $\Omega_{\pm k}(SM;\C^n)$ for all $\pm k\ge 1$. For unitary $A$, Theorem~\ref{th::injectivity of mu+A operators in closed case} and Corollary~\ref{cor::injectivity of mu+A operators for small A} ensure injectivity of these operators under the hypotheses \eqref{eqn::condition 1} and \eqref{eqn::condition 2}, respectively. The proof is complete.
\end{proof}

%--------------------------------------------------------------------------------------------------------------------------------------------------------------------

\section{Injectivity of $\mu_\pm+A_{\pm}$ operators}\label{sec::injectivity of mu+A operators}

In the present section we prove injectivity results for $\mu_\pm+A_{\pm}$ operators which are one of the crucial components in the proofs of Theorem~\ref{th::main 1} and Theorem~\ref{th::main 3}. %These results generalize the corresponding result obtained in \cite{assylbekov2017invariant} for $\mu_\pm$ operators.

\subsection{Surfaces with boundary} This case can be reduced to a similar result for geodesic flows obtained in \cite{guillarmou2016negconnections}.

\begin{Theorem}\label{th::injectivity of mu+A operators in boundary case}
Let $(M,g,E)$ be a Gaussian thermostat on a compact oriented surface with boundary and let $A:TM\to \mathfrak{gl}(n, \C)$ be a connection. If $u\in \Omega_{\pm k}(SM;\C^n)$, $k\ge 1$, satisfy $(\mu_{\pm}+A_{\pm})u=0$ and $u|_{\p(SM)}=0$, then $u=0$.%Then the operator $\mu_{+}+A_{+}:\Omega_k\to\Omega_{k+1}$ is injective for $k\ge 1$ and $\mu_{-}+A_{-}:\Omega_k\to\Omega_{k-1}$ is injective for $k\le-1$.
\end{Theorem}

\begin{proof}%[Proof of Theorem~\ref{th::injectivity of mu+A operators in boundary case}]
%We give the proof only for $\mu_{+}+A_{+}$. Analogous reasonings work for $\mu_{-}+A_{-}$.
Let $\theta$ be a $1$-form on $M$ dual to $E$. Then $\lambda_+=i\theta_1$ and $\lambda_-=-i\theta_{-1}$. Therefore, the operators $\mu_{+}+A_{+}$ and $\mu_{-}+A_{-}$ restricted to $\Omega_{k}(SM;\C^n)$ and $\Omega_{-k}(SM;\C^n)$, respectively, have the forms
$$
\mu_{+}+A_{+}=\eta_{+}+(A_{+}-k\theta_1)\quad\text{and}\quad \mu_{-}+A_{-}=\eta_{-}+(A_{-}-k\theta_{-1}).
$$
Now, fix $k\ge 1$ and consider $A^k:=A-k\theta$ as a connection. Then, according to \cite[Theorem~5.2]{guillarmou2016negconnections}, any $u\in \Omega_{\pm m}(SM;\C^n)$, $m\ge 1$, such that $(\eta_{\pm}+A^k_{\pm})u=0$ and $u|_{\p(SM)}=0$ must vanish identically. In the case $m=k$, this, in particular, yields the desired injectivity result for $(\mu_{\pm}+A_{\pm})$.%Suppose that $u\in \Omega_k(SM;\C^n)$, $k\ge 1$, satisfies $(\mu_{+}+A_{+})u=0$. Take a finite collection $\{U_\alpha,w_k^\alpha\}_\alpha$ constructed in Lemma~\ref{lem::mu = conjugate of eta}. Then $(\eta_{+}+A_{+})(e^{w_k^\alpha}u)=0$ in $S(U_\alpha\cap M)$ for every $\alpha$. Since $u|_{\p(SM)}=0$, using \cite[Theorem~5.2]{guillarmou2016negconnections}, we get $u=0$ in $S(U_\alpha\cap M)$ for all $U_\alpha$ such that $U_\alpha\not\subset M$. Let $N$ be the union of such sets $U_\alpha$. Hence, $u=0$ in $S(N\cap M)$. Next, by \cite[Theorem~5.2]{guillarmou2016negconnections}, we have $u=0$ in $SU_\beta$ for $U_\beta\subset M^{\rm int}$ such that $U_\beta\cap N\neq\varnothing$. Continuing this process, we show that $u=0$ in $S(U_\alpha\cap M)$ for all $\alpha$. Hence, we have shown $u=0$ in $SM$ as desired. Similar proof works for $\mu_{-}+A_{-}$.
\end{proof}

\subsection{Closed surfaces}
This case is more complicated. We closely follow the arguments of \cite[Section~8]{guillarmou2016negconnections}. We remind the reader that on closed surfaces we always work with unitary connections. First, we need the following weighted analog of Proposition~\ref{prop::GK energy identity for mu operators} for $\mu_{\pm}+A_{\pm}$ operators.

\begin{Proposition}\label{prop::weighted GK energy identity for mu+A operators}
Let $(M,g,E)$ be a Gaussian thermostat on a compact oriented surface and let $A:TM\to \mathfrak u(n)$ be a unitary connection. Suppose that $\varphi\in C^\infty(M;\R)$. Then for any $u\in \Omega_k(SM;\C^n)$, $k\in\Z$, with $u|_{\p(SM)}=0$ in the case $\p M\neq \varnothing$, we have
\begin{multline*}
\|e^{-\varphi}(\mu_{+}+A_{+})(e^{\varphi} u)\|^2\\
=\|e^{\varphi}(\mu_{-}+A_{-})(e^{-\varphi} u)\|^2-\frac{k}{2}(K_E u,u)-\frac{1}{2}((\Delta_g\varphi) u,u)+\frac{i}{2}(\star F_A u,u).
\end{multline*}
\end{Proposition}
\begin{proof}
Define the operators
\begin{align*}
P_\varphi&=e^{-\varphi}\circ(\mu_{+}+A_{+})\circ e^{\varphi}=(\mu_{+}+A_{+})+(\eta_{+}\varphi),\\
Q_\varphi&=e^{\varphi}\circ (\mu_{-}+A_{-})\circ e^{-\varphi}=(\mu_{-}+A_{-})-(\eta_{-}\varphi).
\end{align*}
It suffices to prove
\begin{equation}\label{eqn::P*P-Q*Q}
P_\varphi^*P_\varphi u-Q_\varphi^*Q_\varphi u=-\frac{k}{2}K_E u-\frac{1}{2}(\Delta_g \varphi)u+\frac{i}{2}\star F_A u,\qquad u\in \Omega_k(SM;\C^n).
\end{equation}
Using \eqref{eqn::adjoints of mu operators} and the fact that $A$ is unitary, it is easy to check that
$$
P_\varphi^*=-(\mu_{-}+A_{-})+i\lambda_{-}+(\eta_{-}\varphi),\quad Q_\varphi^*=-(\mu_{+}+A_{+})-i\lambda_{+}-(\eta_{+}\varphi).
$$
Then for $u\in \Omega_k(SM;\C^n)$,
\begin{align*}
P_\varphi^*&P_\varphi u=(-(\mu_{-}+A_{-})+i\lambda_{-}+(\eta_{-}\varphi))((\mu_{+}+A_{+})+(\eta_{+}\varphi))u\\
&=-(\mu_{-}+A_{-})(\mu_{+}+A_{+})u-(\eta_{-}\eta_{+}\varphi)u-i\lambda_{-}(\eta_{+}\varphi)u-(\eta_{+}\varphi)(\mu_{-}+A_{-})u\\
&\quad+i\lambda_{-}(\mu_{+}+A_{+})u+i\lambda_{-}(\eta_{+}\varphi)u+(\eta_{-}\varphi)(\mu_{+}+A_{+})u+(\eta_{-}\varphi)(\eta_{+}\varphi)u\\
&=-(\mu_{-}+A_{-})(\mu_{+}+A_{+})u-(\eta_{-}\eta_{+}\varphi)u-(\eta_{+}\varphi)(\mu_{-}+A_{-})u+i\lambda_{-}(\mu_{+}+A_{+})u\\
&\quad+(\eta_{-}\varphi)(\mu_{+}+A_{+})u+(\eta_{-}\varphi)(\eta_{+}\varphi)u
\end{align*}
and
\begin{align*}
Q_\varphi^*&Q_\varphi u=(-(\mu_{+}+A_{+})-i\lambda_{+}-(\eta_{+}\varphi))((\mu_{-}+A_{-})-(\eta_{-}\varphi))u\\
&=-(\mu_{+}+A_{+})(\mu_{-}+A_{-})u+(\eta_{+}\eta_{-}\varphi)u-i\lambda_{+}(\eta_{-}\varphi)u+(\eta_{-}\varphi)(\mu_{+}+A_{+})u\\
&\quad-i\lambda_{+}(\mu_{-}+A_{-})u+i\lambda_{+}(\eta_{-}\varphi)u-(\eta_{+}\varphi)(\mu_{-}+A_{-})u+(\eta_{+}\varphi)(\eta_{-}\varphi)u\\
&=-(\mu_{+}+A_{+})(\mu_{-}+A_{-})u+(\eta_{+}\eta_{-}\varphi)u+(\eta_{-}\varphi)(\mu_{+}+A_{+})u-i\lambda_{+}(\mu_{-}+A_{-})u\\
&\quad-(\eta_{+}\varphi)(\mu_{-}+A_{-})u+(\eta_{+}\varphi)(\eta_{-}\varphi)u.
\end{align*}
Therefore,
\begin{multline*}
P_\varphi^*P_\varphi u-Q_\varphi^*Q_\varphi u=[\mu_{+}+A_{+},\mu_{-}+A_{-}]u-(\eta_{-}\eta_{+}\varphi+\eta_{+}\eta_{-}\varphi)u\\
+i\lambda_{-}(\mu_{+}+A_{+})u+i\lambda_{+}(\mu_{-}+A_{-})u.
\end{multline*}
This gives \eqref{eqn::P*P-Q*Q} by applying Lemma~\ref{lem::commutation formula for mu+A operators} and the fact $\eta_{-}\eta_{+}\varphi+\eta_{+}\eta_{-}\varphi=\frac{1}{2}\Delta_g\varphi$.
\end{proof}

As an immediate consequence of Proposition~\ref{prop::weighted GK energy identity for mu+A operators}, we obtain our first result on injectivity of $\mu_{\pm}+A_{\pm}$ on closed surfaces.

\begin{Corollary}\label{cor::injectivity of mu+A operators for small A}
Let $(M,g,E)$ be a Gaussian thermostat on a closed oriented surface with $K_E\le-\kappa$ for some constant $\kappa>0$  and let $A:TM\to \mathfrak u(n)$ be a unitary connection. Suppose that for all $k\ge 1$,
$$
k>\frac{\|i\star F_A\|_{L^\infty(M)}}{\kappa}.
$$
Then $\mu_{\pm}+A_{\pm}$ is injective on $\Omega_{\pm k}(SM;\C^n)$ for all $k\ge 1$.
\end{Corollary}

Now, we are ready to state our second result on injectivity of $\mu_{\pm}+A_{\pm}$ on closed surfaces which is an analog of \cite[Theorem~8.1]{guillarmou2016negconnections}.

\begin{Theorem}\label{th::injectivity of mu+A operators in closed case}
Let $(M,g,E)$ be a Gaussian thermostat on a closed oriented surface and let $A:TM\to \mathfrak u(n)$ be a unitary connection. Denote by $\lambda_{\min}$ and $\lambda_{\max}$ the smallest and largest, respectively, eigenvalues of $i\star F_A$.
\begin{itemize}
\item[(a)] If $k\in \Z$ and
$$
\int_M\lambda_{\min}\,d\Vol_g>2\pi k\chi(M),
$$
then any $u\in \Omega_k(SM;\C^n)$ satisfying $(\mu_{+}+A_{+})u=0$ vanishes identically.

\item[(b)] If $k\in \Z$ and
$$
\int_M\lambda_{\max}\,d\Vol_g<-2\pi k\chi(M),
$$
then any $u\in \Omega_{-k}(SM;\C^n)$ satisfying $(\mu_{-}+A_{-})u=0$ vanishes identically.
\end{itemize}
\end{Theorem}

\begin{proof}%[Proof of Theorem~\ref{th::injectivity of mu+A operators in closed case}]
We prove (a). Suppose that there is $\varphi\in C^\infty(M;\R)$ and some constant $C>0$ such that
\begin{equation}\label{ineq::condition for Carleman estimate for mu operators}
-kK_E -(\Delta_g \varphi)+i\star F_A\ge 2C\Id\quad\text{on}\quad M
\end{equation}
as positive definite endomorphisms. Then Proposition~\ref{prop::weighted GK energy identity for mu+A operators} below yields
$$
C\|w\|\le \|e^{-\varphi}(\mu_{+}+A_{+})(e^{\varphi} w)\|,\quad w\in\Omega_k(SM;\C^n).
$$
If $(\mu_{+}+A_{+})u=0$, taking $w=e^{-\varphi}u$ in this estimate will give $u=0$ as desired.

To find $\varphi$ satisfying \eqref{ineq::condition for Carleman estimate for mu operators}, we need $f\in C^\infty(M;\R)$ such that
\begin{equation}\label{ineq::conditions for f}
f+\lambda_1>0\quad\text{and}\quad \int_M f\,d\Vol_g=-2\pi k\chi(M).
\end{equation}
Such $f$, by the Gauss-Bonnet theorem, satisfies
$$
\int_M [kK_E+f]\,d\Vol_g=\int_M [kK+f]\,d\Vol_g=0.
$$
Therefore, there exists $\varphi\in C^\infty(M;\R)$ solving the equation $-\Delta_g\varphi=kK_E+f$ in $M$. Using this, one can see that
$$
-kK_E -(\Delta_g \varphi)+i\star F_A=f+i\star F_A\ge f+\lambda_1\ge 2C>0
$$
as required. 

Hence, we need to find $f$ satisfying \eqref{ineq::conditions for f}. To that end, define
$$
\varepsilon:=\frac{1}{\Vol_g(M)}\Big[\int_M \lambda_1\,d\Vol_g-2\pi k\chi(M)\Big]>0.
$$
Since $\lambda_1\in\C(M;\R)$ there is $h\in C^\infty(M;\R)$ such that $\|h-\lambda_1\|\le \varepsilon/4$. Define $f:=-h+\varepsilon_0$, where $\varepsilon_0$ is the constant determined by
$$
\int_M f\,d\Vol_g=-2\pi k\chi(M).
$$
Then $f+\lambda_1=\varepsilon_0-[h-\lambda_1]\ge \varepsilon_0-\varepsilon/4$. However,
\begin{align*}
\varepsilon_0&=\frac{1}{\Vol_g(M)}\int_M [f+h]\,d\Vol_g=\frac{1}{\Vol_g(M)}\Big[\int_M h\,d\Vol_g-2\pi k\chi(M)\Big]\\
&=\varepsilon+\frac{1}{\Vol_g(M)}\int_M [h-\lambda_1]\,d\Vol_g\ge \varepsilon-\frac{\varepsilon}{4}=\frac{3\varepsilon}{4}.
\end{align*}
Hence, we finally get $f+\lambda\ge \varepsilon/2$. The proof of (b) is analogous.
\end{proof}

%--------------------------------------------------------------------------------------------------------------------------------------------------------------------

\section{Injectivity results for the nonlinear problems}\label{sec::nonlinear problems}

In the present section we prove injectivity results for the nonlinear problems, namely Theorem~\ref{th::main 2} and Theorem~\ref{th::main 4}.

\subsection{Surfaces with boundary} 
We use a {\it pseudolinearization} to reduce the nonlinear problem to Theorem~\ref{th::main 1}. This argument was used in earlier works \cite{guillarmou2016negconnections,paternain2018carleman,paternain2012attenuated,paternain2016geodesic,zhou2017generic}.

\begin{proof}[Proof of Theorem~\ref{th::main 2}]
Let $U_{A,\Phi}:SM\to GL(n,\C)$, $U_{B,\Psi}:SM\to GL(n,\C)$ be the unique solutions for the problems
\begin{align*}
(\bfG_E+A+\Phi)U_{A,\Phi}&=0\quad\text{in}\quad SM,\quad U_{A,\Phi}|_{\p_+ SM}=\Id,\\
(\bfG_E+B+\Psi)U_{B,\Psi}&=0\quad\text{in}\quad SM,\quad U_{B,\Psi}|_{\p_+ SM}=\Id,
\end{align*}
respectively. Then $C_{A,\Phi}=U_{A,\Phi}|_{\p_- SM}$ and $C_{B,\Psi}=U_{B,\Psi}|_{\p_- SM}$. It is not difficult to show that $Q:=U_{A,\Phi} U_{B,\Psi}^{-1}: SM\to GL(n,\C)$ satisfies
$$
\bfG_EQ+AQ-QB+\Phi Q-Q\Psi=0\quad\text{in}\quad SM,\quad Q|_{\p(SM)}=\Id.
$$
Set $U:=Q-\Id$. Then $U:SM\to\C^{n\times n}$ solves
$$
\bfG_EU+AU-UB+\Phi U-U\Psi=-(A-B+\Phi-\Psi)\quad\text{in}\quad SM,\quad U|_{\p(SM)}=0.
$$
Let $\mathbf A:TM\to\mathfrak{gl}(n\times n,\C)$ be a connection on the trivial bundle $M\times \C^{n\times n}$ defined as $\mathbf AU:=AU-UB$, and let $\mathbf\Phi:M\to\mathfrak{gl}(n\times n,\C)$ be a Higgs field defined as $\mathbf\Phi U:=\Phi U-U\Psi$. Then we can rewrite the above transport equation as
$$
\bfG_EU+\mathbf AU+\mathbf\Phi U=-(A-B+\Phi-\Psi)\quad\text{in}\quad SM,\quad U|_{\p(SM)}=0.
$$
In other words, $I_{\mathbf A,\mathbf\Phi}[A-B,\Phi-\Psi]=0$. Then by Theorem~\ref{th::main 1}, there is $p\in C^\infty(M;\C^{n\times n})$ with $p|_{\p M}=0$ such that
$$
A-B=d_{\mathbf A}p=dp+Ap-pB,\quad \Phi-\Psi=\mathbf \Phi p=\Phi p-p\Psi.
$$
Uniqueness for solutions of transport equations imply that $U=-p$. Setting back $Q:=U+\Id$, we finish the proof.
\end{proof}

\subsection{Closed surfaces}
Let $(M,g,E)$ be a Gaussian thermostat on a closed oriented surface $M$ with $K_E<0$. Let $(A,\Phi)$ be a pair of a unitary connection and a skew-Hermitian Higgs field and let $C_{A,\Phi}:SM\times\R\to U(n)$ be the corresponding cocycle.

\begin{Definition}{\rm
The pair $(A,\Phi)$ will be called \emph{cohomologically trivial} if there is $u\in C^\infty(SM;U(n))$ such that
$$
C_{A,\Phi}(x,v;t)=u(\phi_t(x,v))u(x,v)^{-1},
$$
In this case, $u$ is referred to as a \emph{trivializing function}.
}\end{Definition}

Clearly, cohomologically trivial pairs are transparent. The opposite does not need to be true in general. However, in our setting, transparent pairs are always cohomologically trivial. Indeed, as mentioned before, the assumption $K_E<0$ guarantees that the thermostat flow $\phi_t$ is Anosov and transitive. Then cohomologically triviality of transparent pairs follows from the Livsic theorem for $U(n)$-cocycles \cite{livvsic1971certain,livvsic1972cohomology} combined with regularity results in \cite{nicticua1998regularity}.

Now we are in position to prove Theorem~\ref{th::main 4}.

\begin{proof}[Proof of Theorem~\ref{th::main 4}]
According to the above discussion, there is a trivializing function $u\in C^\infty(SM;U(n))$. It is easy to see that $u$ satisfies
\begin{equation}\label{eqn::transport equation for trivializing functions}
(\bfG_E+A+\Phi)u=0\quad\text{in}\quad SM.
\end{equation}
By Theorem~\ref{th::main 3}, we deduce that $u\in C^\infty(M;U(n))$. Then \eqref{eqn::transport equation for trivializing functions} becomes $du+Au+\Phi u=0$, which can be separated as $du+Au=0$ and $\Phi u=0$. Setting $Q:=u\in C^\infty(M;U(n))$, this gives $Q^{-1}(d+A)Q=0$ and $\Phi=0$.
\end{proof}

%--------------------------------------------------------------------------------------------------------------------------------------------------------------------

\appendix
\section{Gaussian thermostats on manifolds with boundary}%\label{sec::appendix}
Throughout this section, $(M,g,E)$ is a non-trapping Gaussian thermostat on a compact manifold with boundary.

\subsection{Convexity}\label{sec::convexity}
Consider a manifold $M_1$ whose interior contains $M$. We extend the metric $g$ and the external field $E$ to $M_1$ smoothly. We use the same notations for extensions. We say that $(M,g,E)$ is \emph{convex} at $x\in\p M$ if $x$ has a neighborhood $U\subset M_1$ such that all thermostat geodesics in $U$, passing through $x$ and tangent to $\p M$ at $x$, lie in $M_1\setminus M^{\rm int}$. Furthermore, $(M,g,E)$ is said to be \emph{strictly convex} at $x$ if these thermostat geodesics do not intersect $M$ except at $x$.

\begin{Lemma}
The following holds if $(M,g,E)$ is convex at $x\in\p M$
\begin{equation}\label{ineq::convexity at x}
\Lambda(x,v)\ge\<E(x),\nu(x)\>_{g(x)}\quad\text{for all}\quad v\in S_x(\p M).
\end{equation}
Moreover, if this inequality is strict then $(M,g,E)$ is strictly convex at $x$.
\end{Lemma}
\begin{proof}
For sufficiently small $U$, take $\rho\in C^\infty(U;\R)$ such that $|\nabla \rho|_g\equiv 1$ and $\p M\cap U=\rho^{-1}(0)$. Then the convexity of $(M,g,E)$ at $x$ implies that $\rho(\gamma_{x,v}(t))\le 0$ for $v\in S_x(\p M)$ and small enough $t\in\R$. Since $\rho(\gamma_{x,v}(t))$ (which is considered as a function of $t$) attains its maximum value at $t=0$,
$$
\frac{d \rho(\gamma_{x,v}(t))}{dt}\Big|_{t=0}=0 \quad\text{and}\quad \frac{d^2 \rho(\gamma_{x,v}(t))}{dt^2}\Big|_{t=0}\le 0.
$$
As in the proof of Lemma~\ref{lem::tau is smooth on influx bundle}, it is straightforward to check that
\begin{align*}
\frac{d \rho(\gamma_{x,v}(t))}{dt}\Big|_{t=0}&=\<\nu(x),v\>_{g(x)},\\
\frac{d^2 \rho(\gamma_{x,v}(t))}{dt^2}\Big|_{t=0}&=\Hess_x\rho(v,v)+\<\nu(x),E(x)\>_{g(x)}.
\end{align*}
Since $x\in\p M$ and $v\in S_x(\p M)$, we always have $\<\nu(x),v\>_{g(x)}=0$ and $\Hess_x\rho(v,v)=-\Lambda(x,v)$. Hence, we get \eqref{ineq::convexity at x}.

Now, suppose that the inequality \eqref{ineq::convexity at x} is strict. Then there is $\delta>0$ such that for every $v\in S_x M$,
$$
\frac{d^2 \rho(\gamma_{x,v}(t))}{dt^2}\Big|_{t=0}\le-\delta.
$$
Therefore, there is a sufficiently small $\varepsilon>0$ such that $\rho(\gamma_{x,v}(t))\le-\delta t^2/4$ for all $t\in(-\varepsilon,\varepsilon)$. This implies the strict convexity at $x$.
\end{proof}

\subsection{Scattering relation and folds}\label{sec::folds}
In this section we prove the following result.

\begin{Theorem}\label{th::characterization of smooth functions constant along the flow}
Let $(M,g,E)$ be a non-trapping and strictly convex Gaussian thermostat on a compact manifold with boundary. We have the following characterization:
$$
C^\infty_\alpha(\p_+SM;\C^n):=\{w\in C^\infty(\p_+SM;\C^n):\mathcal Aw\in C^\infty(\p(SM);\C^n)\}.
$$
\end{Theorem}

For the proof, we use the notion of a Whitney fold. % following \cite[Definition~C.4.1]{hormander1985analysisIII}.

\begin{Definition}{\rm
Let $\mathcal M$ and $\mathcal N$ be smooth manifolds having the same dimension and let $f:\mathcal M\to\mathcal N$ be a smooth map. We say that $f$ is a \emph{Whitney fold}, with fold $L\subset \mathcal M$, at $m\in L$ if $\{x\in\mathcal M: d_x f\text{ is singular}\}$ is a smooth hypersurface near $m$ and $\ker(d_m f)$ is transverse to $T_m L$.% We say that $f$ \emph{has a fold on} $L\subset\mathcal M$ if it has a fold at every point of $L$.
}\end{Definition}

Consider  a closed manifold $N$ containing $M$. Extend the metric $g$ and the vector field $E$ smoothly to $N$, preserving the former notations for extensions. Next, embed $M$ into the interior of a compact manifold $M_1\subset N$ with boundary. Choose $M_1$ to be sufficiently close to $M$ so that $(M_1,g,E)$ is also non-trapping and strictly convex.

For $(x,v)\in SM_1$, we denote by $\ell(x,v)$ the first non-negative time such that the thermostat geodesic $\gamma_{x,v}$ exits $M_1$.

\begin{Lemma}\label{lem::Psi have a fold}
The map $\Psi:\p(SM)\to\p_- SM_1$ defined as
$$
\Psi(x,v):=\big(\gamma_{x,v}(\ell(x,v)),\dot\gamma_{x,v}(\ell(x,v))\big),\quad (x,v)\in \p(SM),
$$
is a Whitney fold with fold $S(\p M)$.
\end{Lemma}
\begin{proof}
Take a smooth $\rho:M\to [0,\infty)$ such that $\rho^{-1}(0)=\p M$ and $|\nabla \rho|_{g}\equiv 1$ near $\p M$. Note that $\nabla\rho=\nu$. Then, as in the proof of Lemma~\ref{lem::tau is smooth on influx bundle}, it is straightforward to check that
$$
\bfG_E(\rho)(x,v)=\<\nu(x),v\>_{g(x)},\quad \bfG_E^2(\rho)(x,v)=\Hess_x\rho(v,v)+\<\nu(x),E(x)\>_{g(x)}
$$
for $(x,v)\in\p(SM)$. If we take $(x.v)\in S(\p M)$, then $\bfG_E(\rho)(x,v)=0$ and $\bfG_E^2(\rho)(x,v)\neq 0$ by the strict convexity assumption. We are in the same situation as in the proof of \cite[Lemma~7.5]{dairbekov2007boundary}. Repeating the arguments therein, we complete the proof.
\end{proof}

Now, we are ready to prove Theorem~\ref{th::characterization of smooth functions constant along the flow}.

\begin{proof}[Proof of Theorem~\ref{th::characterization of smooth functions constant along the flow}]
If $w_\psi\in C^\infty_\alpha(\p_+ SM;\C^n)$, then we have $\mathcal Aw=w_\psi|_{\p(SM)}\in C^\infty(\p(SM);\C^n)$. To prove the opposite implication, suppose $\mathcal Aw\in C^\infty(\p(SM);\C^n)$. Then Lemma~\ref{lem::Psi have a fold} and \cite[Theorem~C.4.4]{hormander1985analysisIII} imply the existence of a smooth function $u$ on a neighborhood of $\Psi(\p(SM))\subset \p_- SM_1$ such that $w=u\circ \Psi$.

Now, consider the maps $\Phi:SM\to\p_-SM$, $\Phi_{M_1}:SM_1\to\p_-SM_1$ given by
\begin{align*}
\Phi(x,v):&=\big(\gamma_{x,v}(\tau(x,v)),\dot\gamma_{x,v}(\tau(x,v))\big),\quad (x,v)\in SM,\\
\Phi_{M_1}(x,v):&=\big(\gamma_{x,v}(\ell(x,v)),\dot\gamma_{x,v}(\ell(x,v))\big),\quad (x,v)\in SM_1,
\end{align*}
respectively. Since $w_\psi=w\circ \mathcal S^{-1}\circ\Phi$, we have $w_\psi=u\circ \Psi\circ \mathcal S^{-1}\circ\Phi$. Observe also that $\Phi_{M_1}|_{SM}=\Psi\circ \mathcal S^{-1}\circ\Phi$. Therefore, we can write $w_\psi=u\circ \Phi_{M_1}|_{SM}$. Then the smoothness of $\Phi_{M_1}$ on $SM$ implies that $w_\psi$ is smooth on $SM$.
\end{proof}

%\subsection{Curvature of a Gaussian thermostats on a surface}

\end{document}